\documentclass[12pt]{amsart}
\usepackage{amssymb,amsmath,amsthm}
\def \e{{\rm e}}

\def \eps{\varepsilon}

\def \d{{\boldsymbol{\delta}}}
\def \dist{{\rm d}}
\def \dd{{\rm d}}
\def \R {{\mathbb R}}
\def \N {{\mathbb N}}

\def \K {{\mathbb K}}
\def \U {{\mathcal O}}
\def\cbt {{\sc cbt}}

\def \au {\rm}
\def \ti {\it}
\def \jou {\rm}
\def \bk {\it}
\def \no#1#2#3 {{\bf #1} (#3), #2.}
\def \eds#1#2#3 {#1, #2, #3.}

\newtheorem{proposition}{Proposition}
\newtheorem{theorem}[proposition]{Theorem}
\newtheorem{corollary}[proposition]{Corollary}
\newtheorem{lemma}[proposition]{Lemma}
\theoremstyle{definition}
\newtheorem{definition}[proposition]{Definition}
\newtheorem*{example}{Example}
\newtheorem*{remark}{Remark}

\title[Totally dissipative dynamical processes]
{Totally dissipative dynamical processes\\
and their uniform global attractors}

\author[V.V. Chepyzhov, M. Conti and V. Pata]
{Vladimir V. Chepyzhov, Monica Conti and Vittorino Pata}

\dedicatory{Dedicated to the Memory of Professor Mark Iosifovich Vishik}

\thanks{Work partially supported by the Russian Foundation of Basic Researches (projects 11-01-00339
and 12-01-00203) and the Russian Ministry of Education and Science (agreement no.\ 8502).}

\address{
Institute for Information Transmission Problems - Russian Academy of Sciences
\newline\indent
Bolshoy Karetniy 19, Moscow 101447, Russia
{\vskip.3mm}
National Research University Higher School of Economics
\newline\indent
Myasnitskaya Street 20, Moscow 101000, Russia}
\email{chep@iitp.ru {\rm (V.V. Chepyzhov)}}

\address{
Politecnico di Milano - Dipartimento di Matematica ``F.\ Brioschi''
\newline\indent
Via Bonardi 9, 20133 Milano, Italy}
\email{monica.conti@polimi.it {\rm (M. Conti)}}
\email{vittorino.pata@polimi.it {\rm (V. Pata)}}


\begin{document}

\begin{abstract}
We discuss the existence of the global attractor for a family of processes $U_\sigma(t,\tau)$ acting on
a metric space $X$ and depending on a symbol $\sigma$ belonging to some other metric space $\Sigma$.
Such an attractor
is uniform with respect to $\sigma\in\Sigma$, as well as
with respect to the choice of the initial time $\tau\in\R$. The existence of the attractor is established
for totally dissipative processes
without any continuity assumption. When the process satisfies some additional (but rather mild)
continuity-like hypotheses, a characterization of the attractor is given.
\end{abstract}

\maketitle

\section{Introduction}

\noindent
Let $(X,\dist)$ be a metric space, not necessarily complete.
A family of maps
$$U(t,\tau ):X\rightarrow X$$
depending on two real
parameters $t\geq\tau$ is
said to be a {\it dynamical process}, or more simply a {\it process}, on $X$ whenever
\begin{itemize}
\item[$\bullet$] $U(\tau ,\tau )={\rm id}_X$ (the identity map in $X$) for all $\tau \in\R$;
\smallskip
\item[$\bullet$] $U(t,\tau )=U(t,s)U(s,\tau )$ for all $t\geq s\geq \tau $.
\end{itemize}
Such a notion is particularly useful to describe the solutions to nonautonomous differential equations
in normed spaces. Indeed, assume to have the equation
\begin{equation}
\label{intro}
\frac{\dd}{\dd t}u(t)={\mathcal A}(t,u(t)),
\end{equation}
where, for every fixed $t\in\R$, ${\mathcal A}(t,\cdot)$ is a (possibly nonlinear) densely defined operator on
a normed space $X$.
If the Cauchy problem for \eqref{intro} is well posed, in some weak sense,
for all times $t\geq \tau$ and all initial data
$u_0\in X$ taken at the initial time $\tau\in\R$, then the corresponding solution $u(t)$
on the time-interval $[\tau,\infty)$
with $u(\tau)=u_0$ reads
$$u(t)=U(t,\tau)u_0,$$
where $U(t,\tau)$ is uniquely determined by the equation, and it is easily seen to satisfy
the two properties above. Within this picture, autonomous systems are just a particular case,
occurring when the operator ${\mathcal A}(t,\cdot)$ is constant in time. In that situation,
the evolution depends only on the difference $t-\tau$. In other words, the equality
$$U(t,\tau)=U(t-\tau,0)$$
holds for every $t\geq\tau$, and the one-parameter family of maps
$$S(h)=U(h,0),\quad h\geq 0,$$
fulfills the semigroup
axioms, i.e.\
\begin{itemize}
\item[$\bullet$] $S(0)={\rm id}_X$;
\smallskip
\item[$\bullet$] $S(h+r)=S(h)S(r)$ for all $h,r\geq 0$.
\end{itemize}
Summarizing, we may say that dynamical processes extend the concept of
dynamical semigroups for the evolution of open models
where time-dependent external excitations are present.

When dealing with differential problems arising from concrete evolutionary phenomena,
we are usually in presence of some dissipation mechanism. Adopting a global-geometrical point of view,
the theory of dissipative dynamical systems describes
this situation in terms of small sets of the phase space able to attract in a suitable
sense the trajectories arising from bounded regions of initial data.
In particular, it is interesting to locate the smallest set where
the whole asymptotic dynamics is eventually confined. For autonomous systems,
such a set is called the global attractor (we address the reader to
the classical books \cite{BV,HAL,HAR,TEM} for more details).
A similar concept can be used in connection with nonautonomous systems. Namely, it is possible to extend
the notion of global attractor for dynamical processes, or even families of dynamical processes
(see \cite{CV1,CV2,CV,HAR,SY}).
In this note, paralleling what done in \cite{CCP} for the semigroup
case, we aim to reconsider the theory of global attractors for families of dynamical processes
by defining the basic objects (e.g.\ the attractor) only in term of their attraction properties,
without appealing to any continuity-like notion (see Sections~{2-4}).
In fact, imposing further conditions on the processes, but still much weaker than continuity,
it is possible to recover the classical characterization given in \cite{CV} of the attractor in terms
of kernel sections of complete bounded trajectories (see Sections~{5-6}).
In the final Sections~7-8, we discuss an application to nonautonomous differential equations.

\subsection*{Notation}
For every $\eps>0$,
the $\eps$-neighborhood of a set $B\subset X$ is defined as
$$\U_\eps(B)=\bigcup_{x\in B}\,\big\{\xi\in X:\, \dist(x,\xi)<\eps\big\}.$$
We denote the standard Hausdorff semidistance
of two (nonempty) sets $B,C\subset X$ by
$$\d_X(B,C)=\sup_{x\in B}\,\dist(x,C)=\sup_{x\in B}\,\inf_{\xi\in C}\,\dist(x,\xi).$$
We have also the equivalent formula
$$\d_X(B,C)=\inf\big\{\eps>0:\, B\subset \U_\eps(C)\big\}.$$

\section{Families of Dissipative Processes}

\noindent
Rather than a single process, given another metric space $\Sigma$
we will consider a family of processes
$$\{U_{\sigma}(t,\tau )\}_{\sigma\in\Sigma}.$$
The parameter $\sigma $ is called
the {\it symbol} of $U_{\sigma }(t,\tau)$, whereas
$\Sigma $ is said to be the {\it symbol space}.
A single process $U(t,\tau )$ can be clearly
viewed as a family of processes with a symbol space made of one element.
Usually, in connection with nonautonomous differential
problems, the symbol is the
collection of all explicitly time-dependent terms appearing in the equations (see \cite{CV1,CV2,CV}).
Besides, the symbol itself may evolve in time in such a way that,
in combination with the process evolution, gives rise to an autonomous
dynamical system acting on the extended phase space $X\times \Sigma$, called
the {\it skew-product flow} or {\it skew-product semigroup} (see Section 4). The study of this semigroup
gives essential information on the evolution of the original
family of processes.

Analogously to the autonomous case, we introduce a number of definitions
that extend the concept of dissipation to the more general nonautonomous
situation. In what follows, the word {\it uniform} will be always understood
with respect to $\sigma\in\Sigma$.

\begin{definition}
A set $B\subset X$ is {\it uniformly absorbing} for
$U_{\sigma}(t,\tau)$ if for every bounded set $C\subset X$ there exists
a (uniform) entering time $t_{\e}=t_{\e}(C)$ such that
$$
U_{\sigma }(t,\tau )C\subset B,\quad\forall\sigma\in\Sigma,
$$
whenever $t-\tau\geq t_\e$.
\end{definition}

The existence of bounded absorbing sets translates in mathematical terms
the fact that a given system is dissipative.

\begin{definition}
The family of processes $U_{\sigma }(t,\tau )$
is said to be {\it uniformly dissipative} if there is a bounded
uniformly absorbing set.
\end{definition}

Uniform dissipativity is however a rather poor notion of dissipation
for a process, unless one can prove the existence of reasonably small (e.g.\ compact
or even of finite fractal dimension)
uniformly absorbing sets. In concrete differential problems,
this is out of reach if, for instance, the system
exhibits some hyperbolicity, which prevents the regularization of initial data.
Hence, a weaker object, albeit more effective, than a uniformly absorbing set should be considered
in order to depict the confinement of the longterm dynamics.

\begin{definition}
A set $K\subset X$ is called a {\it uniformly $\eps$-absorbing set}, for some $\eps>0$,
if its $\eps$-neighborhood $\U_\eps(K)$ is a uniformly absorbing set.
If $K$ is a uniformly $\eps$-absorbing set for all $\eps>0$,
then it is called a {\it uniformly attracting set}.
\end{definition}

The latter definition can be more conveniently given in terms
of Hausdorff semidistance: a set $K$
is uniformly attracting if for any bounded set $C\subset X$ we have
the limit relation
\begin{equation}
\label{attr-prop}
\lim_{t-\tau \to \infty }\,\Big[\sup_{\sigma \in \Sigma}\,
\d_X(U_{\sigma }(t,\tau )C,K)\Big]=0.
\end{equation}

\begin{remark}
Actually, the definition of a uniformly attracting set given in \cite{CV} (as well
as the one of a uniformly
absorbing set) is a little
different: indeed, the limit relation \eqref{attr-prop} is there replaced by
$$
\lim_{t\to \infty }\,\Big[\sup_{\sigma \in \Sigma}\,
\d_X(U_{\sigma }(t,\tau )C,K)\Big]=0,
$$
for every {\it fixed} $\tau\in\R$. Compared to this one, the definition
\eqref{attr-prop} adopted in the present paper is uniform with respect to $\tau\in\R$,
which renders the notion of attraction slightly stronger, and more closely related
to the concrete examples arising from partial differential equations.
Nonetheless, all the results proved in this paper remain valid (with the same proofs)
in the framework of \cite{CV}.
\end{remark}

Let now $\mathfrak{C}_{\Sigma }$ denote the collection of all possible
sequences in $X$ of the form
$$
y_{n}=U_{\sigma _{n}}(t_{n},\tau _{n})x_{n},
$$
where $x_{n}\in X$ is a bounded sequence, $\sigma_{n}\in\Sigma$ and
$t_{n}-\tau _{n}\to\infty$.
For any $y_{n}\in
\mathfrak{C}_{\Sigma },$ we consider the set
$$
L_{\Sigma }(y_{n})=\big\{x\in X:\,y_{n}\to x\ \text{up to a subsequence}\big\},
$$
which in fact can be empty for some $y_{n}$, or $y_n$ may contain a subsequence $y_{n_\imath}$ such that
$L_{\Sigma }(y_{n_\imath})=\emptyset$. Accordingly, we can rephrase the
attraction property \eqref{attr-prop} as follows.

\begin{lemma}
\label{AP}
A set $K\subset X$ is uniformly attracting for the family $U_{\sigma }(t,\tau )$ if and only if
$$
\dist(y_{n},K)\to 0,\quad \forall y_{n}\in \mathfrak{C}_{\Sigma}.
$$
\end{lemma}

Next, we denote the union of all
$L_{\Sigma }(y_{n})$ by
$$
A_{\Sigma }^{\star}=\left\{x\in X:\, y_{n}\to x\ \text{up
to a subsequence, for some }y_{n}\in \mathfrak{C}_{\Sigma }\right\},
$$
and, for any given bounded set $C\subset X$, we define the
{\it uniform $\omega$-limit} of $C$ by
\begin{equation}
\label{omega}
\omega_{\Sigma}(C)=\bigcap_{h\geq 0}\,\overline{\bigcup_{\sigma\in\Sigma}
\,\bigcup_{t-\tau \geq h}\,U_{\sigma }(t,\tau )C},
\end{equation}
the {\it bar} standing for closure in $X$.
Note that,
without additional assumptions, both the sets
$A_{\Sigma }^{\star}$ and $\omega_{\Sigma}(C)$ might be empty.

\begin{lemma}
\label{DP}
The following assertions hold:
\begin{itemize}
\item[(i)] $A_{\Sigma}^{\star}$ is contained in any closed
uniformly attracting set.
\smallskip
\item[(ii)] For any bounded set $C\subset X$ we have the inclusion
$\omega_{\Sigma }(C)\subset A_{\Sigma }^{\star}$. Besides,
$$A_{\Sigma}^{\star}=\bigcup\,\omega_{\Sigma}(C),$$
where the union is taken over all bounded sets $C\subset X$.
\smallskip
\item[(iii)] If $U_{\sigma}(t,\tau)$ is uniformly
dissipative, then for any bounded uniformly absorbing set $B$ and any bounded set $C$ we have the relation
$$\omega_\Sigma(C)\subset\omega_\Sigma(B)=A_{\Sigma}^{\star}.$$
The latter equality in particular implies that
$A_{\Sigma}^{\star}$ is closed in $X$.
\end{itemize}
\end{lemma}

\begin{proof}
If $x\in A_\Sigma^\star$,
then $y_{n}\to x$ for some $y_{n}\in \mathfrak{C}_{\Sigma}$, and we readily obtain from
Lemma \ref{AP} that
$$
\dist(y_{n},K)\to 0
$$
for any uniformly attracting set $K$. If $K$ is also closed, then
$x\in K$, which proves~(i). Concerning (ii), the inclusion $\omega_{\Sigma }(C)\subset A_{\Sigma }^{\star}$
is straightforward from \eqref{omega}, whereas the subsequent equality comes from the very definition of $L_\Sigma(y_n)$.
Indeed,
$$x\in L_\Sigma(y_n)\quad\Rightarrow\quad x\in\omega_\Sigma(\{x_n\}),$$
where $x_n$ is a bounded sequence in $X$ and
$$y_n=U_{\sigma_n}(t_n,\tau_n)x_n\in \mathfrak{C}_{\Sigma}.$$
Finally, in light of (ii), the equality $\omega_\Sigma(B)=A_{\Sigma}^{\star}$
in (iii) clearly follows from
the inclusion $\omega_\Sigma(C)\subset\omega_\Sigma(B)$.
Let then $B$ and $C$ be a bounded uniformly absorbing set and a
bounded set, respectively. Then there is $t_\e>0$ such that
$$
U_{\sigma }(\tau +t_\e,\tau )C\subset B,\quad  \forall \tau \in\R,\,\forall \sigma \in \Sigma.
$$
Therefore, for $t-\tau\geq t_\e$ we get
$$
U_{\sigma }(t,\tau )C =U_{\sigma }(t,\tau +t_\e)U_{\sigma }(\tau +t_\e,\tau)C
\subset U_{\sigma}(t,\tau +t_\e)B.
$$
Accordingly,
$$
\bigcup_{t-\tau \geq h+t_\e}\,U_{\sigma }(t,\tau )C
\subset \bigcup_{t-\tau \geq h}\,U_{\sigma }(t,\tau )B,
$$
and taking the union over all $\sigma \in \Sigma $ and the intersection in
$h\geq 0$,  from \eqref{omega} we arrive at the desired inclusion.
\end{proof}

Among uniformly attracting sets, of particular interest are the compact ones.
Hence, following \cite{CCP}, we consider the
collection of sets
$$
\K_{\Sigma}=\big\{K\subset X:\,K\text{ is compact and uniformly attracting}
\big\}.
$$
Using the results above, we establish a necessary and sufficient condition
in order for a compact set to
be uniformly attracting.

\begin{proposition}
\label{CPT}
Let $K\subset X$ be a compact set. Then $K\in\K_{\Sigma}$ if and
only if
$$
\emptyset \neq L_{\Sigma}(y_{n})\subset
K,\quad \forall y_{n}\in \mathfrak{C}_{\Sigma}.
$$
\end{proposition}

\begin{proof}
If $K$ is uniformly attracting and $y_n\in\mathfrak{C}_{\Sigma}$,
point (i) of Lemma \ref{DP} implies that $L_{\Sigma}(y_n)\subset K$.
Besides, by Lemma \ref{AP},
$$\dist(y_n,\xi_n)\to 0,$$
for some $\xi_n\in K$. Since $K$ is compact, there is $\xi\in K$
such that (up to a subsequence)
$$\xi_n\to\xi\in K\quad\Rightarrow\quad y_n\to\xi\quad\Rightarrow\quad L_{\Sigma}(y_n)\neq\emptyset.$$
Conversely,
if $K$ is not attracting,
$$\dist(y_n,K)>\eps,$$
for some $\eps>0$ and $y_n\in\mathfrak{C}_{\Sigma}$. Therefore, $L_{\Sigma}(y_n)\cap K=\emptyset$.
\end{proof}

As a straightforward consequence, we deduce a corollary.

\begin{corollary}
\label{CPTcor}
If $K_1,K_2\in\K_{\Sigma}$ then
$K_1\cap K_2\in\K_{\Sigma}$.
\end{corollary}

As it will be clear in the subsequent section, the collection $\K_{\Sigma}$ plays a crucial role
in the asymptotic analysis of the process. This motivates the
following definition.

\begin{definition}
The family $U_{\sigma }(t,\tau)$ is
called {\it uniformly asymptotically compact} if it has a compact
uniformly attracting set, i.e.\ if the collection $\K_{\Sigma}$ is
nonempty.
\end{definition}

\begin{remark}
It is apparent that any uniformly asymptotically compact family is
in particular uniformly dissipative.
\end{remark}

\begin{proposition}
\label{proproia}
If $U_{\sigma }(t,\tau)$ is
uniformly asymptotically compact, then
$A_{\Sigma}^{\star}\in \K_{\Sigma}$.
\end{proposition}

\begin{proof}
By assumption, there exists $K\in \K_{\Sigma}$.
Due to Proposition \ref{CPT}, $L_\Sigma(y_n)\neq\emptyset$
for all $y_n\in\mathfrak{C}_\Sigma$ and $A_{\Sigma }^{\star}$ is not empty, being
the union of all $L_\Sigma(y_n)$.
Besides, $A_{\Sigma }^{\star}\subset K$ by (i) of Lemma \ref{DP}.
Since the family $U_{\sigma}(t,\tau)$ is uniformly dissipative,
by (iii) of Lemma \ref{DP} we learn that $A_{\Sigma }^{\star}$ is closed,
and since $K$ is compact, $A_{\Sigma }^{\star}$ is compact as well.
Finally, invoking again Proposition \ref{CPT}, we conclude that
$A_{\Sigma }^{\star}$ is uniformly attracting.
\end{proof}

\section{Uniform Global Attractors}

\noindent
Like in the autonomous case, we are interested in finding the minimal compact
attracting set. In fact, dealing with nonautonomous systems, the property of
minimality turns out to be the natural one
to define the (unique) global attractor,
since we cannot rely any longer on the invariance property, typical
of semigroups.
Hence, the following definition sounds even
more reasonable in the nonautonomous framework.

\begin{definition}
A compact set $A_{\Sigma}\subset X$ is said to be the
{\it uniform global attractor} of the family of processes
$U_{\sigma }(t,\tau )$ if it is uniformly attracting and
is contained in
any compact uniformly attracting
set.
\end{definition}

According to the previous discussion, the attractor $A_\Sigma$ is also uniform with respect
to the choice of the initial time $\tau\in\R$.

\begin{remark}
It is actually possible to develop a theory of global attractors for locally
asymptotically compact semigroups (or processes), which cannot be
dissipative in the traditional sense (i.e.\ existence of a bounded absorbing set).
Still, one can prove
the existence of a unique locally compact global attractor.
In this case, the global attractor is defined to be the smallest closed (instead of compact) attracting set
(see \cite{CV1,CV2,CV}).
\end{remark}

\begin{proposition}
The family $U_{\sigma }(t,\tau )$ possesses at
most one uniform global attractor.
\end{proposition}

\begin{proof}
By contradiction, suppose not. Then, by virtue of Corollary \ref{CPTcor}, the intersection
of two different uniform global attractors belongs to
$\K_{\Sigma}$, which contradicts the minimality property.
\end{proof}

\begin{remark}
For any $\Sigma_0\subset\Sigma$ we have the inclusion $A_{\Sigma_0}\subset A_{\Sigma}$
where $\Sigma_0$ is the symbol space of the subfamily
$\{U_{\sigma }(t,\tau )\}_{\sigma\in\Sigma_0}$
and $A_{\Sigma_0}$ is the uniform global
attractor of this subfamily. In particular,
$A_{\{\sigma\}}\subset A_{\Sigma}$ for any fixed $\sigma\in\Sigma$.
\end{remark}

The main existence result
for the uniform global
attractor reads as follows.

\begin{theorem}
\label{uac}
If the family $U_{\sigma}(t,\tau)$ is uniformly
asymptotically compact, then it possesses the global attractor $A_{\Sigma}$
which coincides with the set $A_{\Sigma }^{\star}$.
\end{theorem}

\begin{proof}
By Proposition \ref{proproia}, we already know that $A_{\Sigma}^{\star}\in \K_{\Sigma}$.
Then, we infer from point (i) of Lemma \ref{DP} that $A_{\Sigma }^{\star}$
is contained in any compact
uniformly attracting set, and hence it is the uniform global attractor.
\end{proof}

Therefore, having a concrete family of processes, the main problem is to
construct at least one compact uniformly attracting set. Such a task
can be, in general, extremely difficult. However,
if the underlying metric space $X$ is complete, there is a more effective way
to express asymptotic compactness. We need first a definition.

\begin{definition}
The family $U_{\sigma }(t,\tau )$
is called {\it uniformly $\eps$-dissipative} if there
exists a finite $\eps$-absorbing set. If the family is
uniformly $\eps$-dissipative for all $\eps>0$,
then it is called {\it totally uniformly dissipative}.
\end{definition}

\begin{remark}
It is readily seen that the family $U_{\sigma }(t,\tau )$
is totally uniformly dissipative if and only if
there is a bounded uniformly absorbing set $B$
for which
$$
\lim_{t-\tau \to \infty }\,\Big[\sup_{\sigma \in \Sigma}\,
\alpha(U_{\sigma }(t,\tau )B)\Big]=0,
$$
where
$$\alpha(C)=\inf\big\{d: \text{$C$ has a finite cover of balls of
$X$ of diameter less than $d$}\big\}$$
denotes the
{\it Kuratowski measure of noncompactness}
of a bounded
set $C\subset X$ (see \cite{HAL} for more details on $\alpha$).
\end{remark}

\begin{theorem}
Let $X$ be a complete metric space. Then the family of processes $U_{\sigma}(t,\tau)$
is uniformly asymptotically compact if and only if
it is totally uniformly dissipative.
\end{theorem}

\begin{proof}
If a uniformly attracting set $K$ is compact, then, for any $\eps>0,$
it has an $\eps $-net $M_{\eps }=\{x_{1},\ldots ,x_{N_\eps}\}$ and,
therefore, the finite set $M_{\eps}$ is uniformly $\eps$-absorbing. Thus, the family
$U_{\sigma }(t,\tau )$ is totally uniformly dissipative.
To show the converse implication, for every $\eps>0,$
let $M_\eps$ be a finite set such that $\U_\eps(M_\eps)$ is absorbing.
We denote
$$K=\bigcap_{\eps>0}B_{\eps }\qquad\text{where}\qquad B_\eps=\overline{\U_\eps(M_\eps)}.$$
The set $K$ is clearly
compact since it is closed and each $M_{\eps}$ is a finite $\eps$-net of $K$.
Consider an arbitrary $y_{n}\in \mathfrak{C}_{\Sigma}$.
The sequence $y_{n}$ is totally bounded since, for every $\eps>0,$ the
set $B_{\eps}$ is uniformly absorbing and therefore $y_{n}\in B_{\eps}$
for sufficiently large $n$ (depending on $\eps$). Hence, $y_{n}$ is precompact
and, since $X$ is complete, the set $L_{\Sigma }(y_{n})$ is nonempty.
Moreover, $L_{\Sigma }(y_{n})\subset B_{\eps}$ for each $\eps>0$, hence,
$$L_{\Sigma}(y_{n})\subset \bigcap_{\eps>0}B_{\eps }=K\quad\Rightarrow\quad K\neq\emptyset.$$
By  Proposition \ref{CPT} we conclude
that the compact set $K$ is uniformly attracting, i.e.\  $K\in\K_\Sigma$
and $U_{\sigma }(t,\tau )$ is uniformly asymptotically compact.
\end{proof}

In conclusion, having a family of processes on a complete metric
space, in order to construct global attractors we only need to prove the total
uniform dissipation property. No continuity assumptions on the processes are required.

\begin{remark}
Let $X$ be a Banach space, and let the family of
processes $U_{\sigma }(t,\tau)$ be uniformly
dissipative, with a bounded
uniformly absorbing set $B$. Then a sufficient condition
for $U_{\sigma }(t,\tau)$ to be totally uniformly dissipative is the following:
for every fixed $\eps>0$ there exist a decomposition $X=Y\oplus Z$
with ${\rm dim}(Y)<\infty$
and a time $t_\star>0$ such that
$$
\sup_{\sigma\in\Sigma}\,\sup_{x\in B}\,\|U_{\sigma }(t,\tau)x-\Pi_Y U_{\sigma}(t,\tau )x\|
<\eps
$$
whenever $t-\tau\geq t_\star$,
where $\Pi_Y$ is the canonical projection of $X$ onto $Y$.
In concrete situations, this condition can be verified by means
of a standard Galerkin approximation scheme.
\end{remark}

We conclude the section by discussing the following problem. Assume we are given another
metric space $\Sigma_0\subset \Sigma$.
Assume also that the subfamily of processes $\{U_{\sigma }(t,\tau )\}_{\sigma\in\Sigma_0}$
has a uniformly (with respect to $\sigma\in\Sigma_0$) attracting set $K$.
The question is now which
conditions guarantee that $K$ is uniformly attracting for the whole
family $U_{\sigma }(t,\tau)$.

\begin{proposition}
\label{propmipiace}
Let the embedding $\Sigma_0\subset \Sigma$ be dense,
and suppose that, for every bounded set $C\subset X$, there exists
$t_C\geq 0$ such that the map
$$\sigma\mapsto U_{\sigma}(t,\tau )x:\Sigma\to X$$
is continuous for any fixed $x\in C$ and every  $t-\tau\geq t_C$. Then any
uniformly attracting set $K$ for the subfamily $\{U_{\sigma}(t,\tau)\}_{\sigma\in\Sigma_0}$
is uniformly attracting for $U_{\sigma}(t,\tau)$ as well.
\end{proposition}

\begin{proof}
Let $C\subset X$ be a bounded set, and let $\eps>0$ be arbitrarily fixed.
Since $K$
is uniformly attracting for $\{U_{\sigma}(t,\tau)\}_{\sigma\in\Sigma_0}$,
there is an entering time $t_\e=t_\e(\varepsilon,C)>0$ such that
$$
U_{\sigma}(t,\tau)C\subset \U_\eps(K),\quad\forall\sigma\in\Sigma_0,
$$
whenever $t-\tau\geq t_\e$.
Since $\Sigma_0$ is dense in $\Sigma$, given $\sigma_\star\in\Sigma$ there
is a sequence $\sigma_{n}\in\Sigma_0$ such that $\sigma_{n}\to\sigma_\star$.
In turn, this yields the convergence
$$U_{\sigma_n}(t,\tau )x\to U_{\sigma_\star}(t,\tau)x$$
for any fixed $x\in C$ and $t-\tau\geq t_C$. Consequently,
$$
U_{\sigma_\star}(t,\tau)C\subset \U_{2\eps}(K),
$$
for every $t-\tau\geq t_\star$, where $t_\star=\max\{t_\e,t_C\}$.
This tells that $K$
is actually uniformly attracting for the whole family.
\end{proof}

\begin{corollary}
\label{corry}
Let the hypotheses of Proposition \ref{propmipiace} hold.
Then the global attractors of both families of processes with
symbol spaces $\Sigma_{0}$ and $\Sigma $ coincide.
\end{corollary}

\section{The Skew-Product Semigroup}

\noindent
Throughout the end of the paper, we will consider
a particular but at the same time very typical situation.

\subsection{General assumptions}
\label{GAss}
Let $\Sigma$ be a compact metric space, and let
$$T(h):\Sigma\to\Sigma,\quad h\geq 0,$$
be a semigroup under whose action $\Sigma$ is fully invariant, i.e.\
$$T(h)\Sigma=\Sigma,\quad\forall h\geq 0.$$
Besides, let
the translation property
\begin{equation}
\label{transl}
U_{\sigma }(h+t,h+\tau )=U_{T(h)\sigma }(t,\tau)
\end{equation}
hold for every $\sigma\in\Sigma$ and every $h\geq 0$ and $t\geq \tau$.
In which case (see \cite{CV1,CV2,CV}), it is easy to verify that the map
\begin{equation}
\label{semigeppo}
S(h)(x,\sigma)=(U_\sigma(h,0)x,T(h)\sigma),\quad h\geq 0,
\end{equation}
defines a (skew-product) semigroup acting on the metric space
$$\boldsymbol{X}=X\times\Sigma.$$

\subsection{Global attractors of semigroups}
Before stating the main result of the section, we recall some facts on abstract semigroups.
Let
$$V(h):Y\to Y,\quad h\geq 0,$$
be a semigroup acting on a (not necessarily complete) metric space $Y$.

\begin{definition}
The semigroup $V(h)$ is said to be {\it asymptotically compact} if there exists a compact attracting set, namely,
a compact set $K\subset Y$ such that
$$\lim_{h\to \infty }\,
\d_Y(V(h)C,K)=0,
$$
for every bounded set $C\subset Y$, where $\d_Y$ denotes the Hausdorff semidistance in $Y$.
\end{definition}

The main theorem in \cite{CCP} reads as follows.

\begin{theorem}
\label{THMminimalia}
If the semigroup $V(h)$ is
asymptotically compact, then there exists the minimal (i.e.\ smallest) compact attracting set $A$,
called the {\it global attractor} of $V(h)$.
\end{theorem}

It is worth observing that such a notion of global attractor is  based only on the minimality
with respect to the attraction property, and does not require any continuity on the semigroup.
Indeed, contrary to the classical notion of attractor (see e.g.\ \cite{BV,HAL,HAR,SY,TEM}), $A$ may fail to be
fully invariant under the action of the semigroup (see examples in \cite{CCP}).

\subsection{The theorem}
Let the family $U_{\sigma}(t,\tau)$ be uniformly
asymptotically compact.\footnote{Since in concrete cases
the underlying space $X$ is usually complete, this is the same as uniformly totally dissipative.}
Then, by Theorem \ref{uac}, we know that
$U_{\sigma}(t,\tau)$ has the uniform global attractor $A_{\Sigma}\subset X$.
It is also clear from Theorem \ref{THMminimalia} that the semigroup
$T(h)$ possesses the global attractor which coincides with the whole
phase space $\Sigma$.

\begin{theorem}
\label{MAIN}
Within the assumptions above, the skew-product semigroup $S(h)$ on $\boldsymbol{X}$
has a (unique) global attractor $\boldsymbol{A}$. Besides, we have the equalities
$$\Pi_1\boldsymbol{A}=A_\Sigma\qquad\text{and}\qquad \Pi_2\boldsymbol{A}=\Sigma,$$
where $\Pi_1$ and $\Pi_2$ denote the canonical projections of $\boldsymbol{X}$ onto its components $X$
and $\Sigma$, respectively.
\end{theorem}

\begin{proof}
It is apparent from the definition \eqref{semigeppo} of
skew-product semigroup
that $A_\Sigma\times \Sigma$
is a (compact) attracting set for $S(h)$.
On account of Theorem \ref{THMminimalia},
this implies that $S(h)$
possesses the global attractor
$\boldsymbol{A}$.
Therefore, appealing to the minimality of $A_\Sigma$ and $\boldsymbol{A}$,
it is enough showing that
\begin{equation}
\label{Piuno}
\Pi_1\boldsymbol{A}\in\K_\Sigma\qquad\text{and}\qquad \Pi_2\boldsymbol{A}=\Sigma.
\end{equation}
Indeed, if $\Pi_1\boldsymbol{A}\in\K_\Sigma$ then $\Pi_1\boldsymbol{A}\supset A_\Sigma$.
On the other hand, being $A_\Sigma\times \Sigma$ compact attracting for $S(h)$, we also get
$$\boldsymbol{A}\subset A_\Sigma\times\Sigma\quad\Rightarrow\quad
\Pi_1\boldsymbol{A}\subset \Pi_1(A_\Sigma\times\Sigma)=A_\Sigma.$$
To see \eqref{Piuno}, the compactness of $\Pi_1\boldsymbol{A}$ being obvious,
let $C\subset X$ be bounded.
Then
$$\lim_{h\to\infty}\,\d_{\boldsymbol{X}}(S(h)(C\times \Sigma),\boldsymbol{A})=0.$$
Equivalently, we can write
$$
\sup_{\sigma\in\Sigma}\,\d_X(U_\sigma(h,0)C,\Pi_1\boldsymbol{A}) \to 0
\qquad\text{and}\qquad\d_\Sigma(T(h)\Sigma,\Pi_2\boldsymbol{A}) \to 0.
$$
The second convergence and the full invariance of $\Sigma$
readily yield the equality
$\Pi_2\boldsymbol{A}=\Sigma$.
We are left to prove the attraction property for $\Pi_1\boldsymbol{A}$.
Since $\Sigma$ is fully invariant for $T(h)$, for $h_\star>0$ to be chosen later we know
that, for any fixed $\sigma\in\Sigma$,
$$\sigma=T(h_\star)\sigma_\star\quad\text{for some}\,\,
\sigma_\star\in\Sigma.$$
Hence, exploiting \eqref{transl} and \eqref{semigeppo},
$$U_\sigma(t,\tau)C = U_{T(h_\star)\sigma_\star}(t,\tau)C
=U_{\sigma_\star}(h_\star+t,h_\star+\tau)C = U_{T(h_\star+\tau)\sigma_\star}(t-\tau, 0)C,$$
upon choosing $h_\star\geq -\tau$.
In light of the first convergence above, we conclude that
$$\sup_{\sigma\in\Sigma}\,
\d_X(U_\sigma(t,\tau)C,\Pi_1\boldsymbol{A}) \leq \sup_{\sigma\in\Sigma}\,
\d_X(U_\sigma(t-\tau,0)C,\Pi_1\boldsymbol{A})\to 0$$
as $t-\tau\to\infty$, proving that $\Pi_1\boldsymbol{A}$ is uniformly attracting for $U_\sigma(t,\tau)$.
\end{proof}

\section{Structure of the Attractor}

\noindent
We now proceed to analyze the structure of the uniform global attractor.
In some sense, this amounts to extend the notion of invariance, typical of semigroups, to dynamical processes.
We begin with two definitions.

\begin{definition}
Let $\sigma\in\Sigma$ be fixed. A function $s\mapsto x(s):\R\to X$  is a {\it complete bounded trajectory} ({\cbt})
of $U_\sigma(t,\tau)$ if and only if the set $\{x(s)\}_{s\in\R}$ is bounded in $X$ and
$$x(s)=U_\sigma(s,\tau)x(\tau),\quad \forall s\geq \tau,\,\forall\tau\in\R.$$
\end{definition}

\begin{definition}
For a fixed $\sigma\in\Sigma$, we call {\it kernel} of the single process $U_\sigma(t,\tau)$
with symbol $\sigma$ the collection of all its {\cbt}. The set
$$
{\mathcal K}_\sigma(t)=\big\{x(t):\,x(s)\text{ is a {\cbt} for }U_\sigma(t,\tau)\big\}
$$
is called the {\it kernel section} at time $t\in\R$.
\end{definition}

Within the framework of the previous section,
the following theorem holds.

\begin{theorem}
\label{THMgazzola}
Assume that there exists $h_\star>0$ such that the maps
$$(x,\sigma)\mapsto U_\sigma(h_\star,0)x:\boldsymbol{X}\to X\qquad\text{and}\qquad
\sigma\mapsto T(h_\star)\sigma:\Sigma\to\Sigma$$
are closed.\footnote{Recall that a map $f:Y\to Z$ is
closed if $f(y)=z$ whenever $y_n\to y$ and $f(y_n)\to z$.}
Then the uniform global attractor $A_\Sigma$ of the family $U_\sigma(t,\tau)$
coincides with the set
$${\mathcal K}_\Sigma=\bigcup_{\sigma\in\Sigma}{\mathcal K}_\sigma(0).$$
\end{theorem}

In fact, the sets ${\mathcal K}_\sigma(0)$ in the statement can be replaced by ${\mathcal K}_\sigma(t)$
for any fixed $t\in\R$.

\begin{remark}
Since $\Sigma$ is compact, it is easy to see that $\sigma\mapsto T(h_\star)\sigma$ closed
actually implies that $T(h_\star)\in{\mathcal C}(\Sigma,\Sigma)$.
\end{remark}

\begin{proof}
We preliminary observe that the closedness assumptions of the theorem
imply that the semigroup $S(h)$
defined by \eqref{semigeppo} is also a closed map on $\boldsymbol{X}$ for $h=h_\star$.
This fact, due to a general result from~\cite{CCP}, is enough to ensure that the global attractor $\boldsymbol{A}$ is
fully invariant for $S(h)$.
In which case, it is well known (see e.g.\ \cite{HAR}) that $\boldsymbol{A}$ is characterized as
$$\boldsymbol{A}=\big\{\boldsymbol{x}(0):\,\boldsymbol{x}(s)\text{ is a {\cbt} for }S(h)\big\},
$$
where a {\cbt} for $S(h)$ is a bounded function $s\mapsto \boldsymbol{x}(s):\R\to \boldsymbol{X}$ such that
$$\boldsymbol{x}(h+s)=S(h)\boldsymbol{x}(s),\quad\forall h\geq 0,\,\forall s\in\R.$$
The same characterization clearly applies for the global
attractor $\Sigma$ of $T(h)$.
The proof now proceeds along the lines of Theorem IV.5.1 in \cite{CV}. For completeness, we report the details.

\smallskip
\noindent
$\bullet$ $\Pi_1\boldsymbol{A}\subset{\mathcal K}_\Sigma$.
Indeed, let
$$\boldsymbol{x}(s)=(x(s),\sigma(s))$$
be a {\cbt} of $S(h)$.
By the very definition of $S(h)$, this is the same as saying that $\sigma(s)$ is a {\cbt} of $T(h)$
(in particular, $\sigma(0)\in\Sigma$),
and
$$x(s)=U_{\sigma(\tau)}(s-\tau,0)x(\tau),\quad\forall s\geq \tau,\,\tau\in\R.$$
If $\tau\geq0$, setting $\sigma_0=\sigma(0)$ and using \eqref{transl}, we have the chain of equalities
$$U_{\sigma(\tau)}(s-\tau,0)x(\tau)=U_{T(\tau)\sigma_0}(s-\tau,0)x(\tau)=U_{\sigma_0}(s,\tau)x(\tau).$$
If $\tau<0$, then $T(-\tau)\sigma(\tau)=\sigma_0$ and using \eqref{transl} in the other direction
we end up with
$$U_{\sigma(\tau)}(s-\tau,0)x(\tau)
=U_{T(-\tau)\sigma(\tau)}(s,\tau)x(\tau)=U_{\sigma_0}(s,\tau)x(\tau).$$
This proves that $x(s)$ is a {\cbt} of
$U_{\sigma_0}(t,\tau)$.

\smallskip
\noindent
$\bullet$ $\Pi_1\boldsymbol{A}\supset{\mathcal K}_\Sigma$.
Let $x_0\in {\mathcal K}_\Sigma$. Then, there exist $\sigma_0\in\Sigma$ and a {\cbt} $x(s)$ of
the process
$U_{\sigma_0}(t,\tau)$ such that $x(0)=x_0$.
Since $\Sigma$ is fully invariant, there is a {\cbt} $\sigma(s)$ of $T(h)$
such that $\sigma(0)=\sigma_0$. We must show that $(x(s),\sigma(s))$ is a {\cbt} of $S(h)$.
Indeed, leaning again on the translation property \eqref{transl}, for $s\geq 0$ we get
\begin{align*}
S(h)(x(s),\sigma(s))
&=(U_{\sigma(s)}(h,0)x(s),T(h)\sigma(s))\\
&=(U_{T(s)\sigma_0}(h,0)x(s),\sigma(h+s))\\
&=(U_{\sigma_0}(h+s,s)x(s),\sigma(h+s))=(x(h+s),\sigma(h+s)).
\end{align*}
The case $s<0$ is similar and left to the reader.

\smallskip
\noindent
Since by Theorem \ref{MAIN} we know that $\Pi_1\boldsymbol{A}=A_\Sigma$, the proof is finished.
\end{proof}

\section{Asymptotically Closed Processes}

\noindent
The aim of this section is to extend the characterization Theorem \ref{THMgazzola}
to a more general class of processes.

\subsection{Asymptotically closed semigroups}
We first need a definition and a theorem from \cite{CCP}
about dynamical semigroups.

\begin{definition}
\label{DEFFY}
A semigroup $V(h)$ acting on a metric space $Y$ is said to be
{\it asymptotically closed} if there exists a sequence of times
$0=h _{0}<h_{1}<h_{2}<h_{3}\ldots $ with the
following property: whenever the convergence $V(h_{k})y_{n}\to \eta^{k}\in X$ occurs
as $n\to\infty$
for every $k\in\N$, we have the equalities
$$V(h_{k})\eta^{0}=\eta^{k},\quad\forall k\in\N.$$
\end{definition}

The sequence $h_k$ in the definition may be finite (but of at least two elements).
In fact, if it is made exactly of two elements $h_0=0$ and $h_1>0$,
then we recover the closedness of the map $V(h_1)$.
On the other hand, if $V(h_\star)$ is closed for some $h_\star>0$, it follows
that $V(h)$ is asymptotically closed with respect to the sequence $h_k=kh_\star$.
This shows that asymptotic closedness is a weaker property than closedness in one point.

\begin{remark}
When the metric space $Y$ is compact, by applying a standard diagonalization
method is immediate to verify that,
if $V(h)$ is asymptotically closed with respect to some sequence $h_k$, then
$$V(h_k)\in{\mathcal C}(Y,Y),\quad\forall k\in\N.$$
\end{remark}

The following theorem holds \cite{CCP}.

\begin{theorem}
\label{butano}
Let $V(h)$
have the global attractor $A$. If $V(h)$ is
asymptotically closed, then $A$ is
fully invariant under the action of the semigroup.
\end{theorem}

\subsection{The theorem}
Hereafter, let the general assumptions \ref{GAss} hold.
Firstly, we extend Definition~\ref{DEFFY} to the case of a family of processes.

\begin{definition}
\label{DGM}
The family $U_{\sigma}(t,\tau)$ is said to be
{\it asymptotically closed} if there exists a sequence of times
$0=h _{0}<h_{1}<h_{2}<h_{3}\ldots $ with the
following property:
if
$$\sigma_n\to\sigma \in \Sigma\qquad\text{and}\qquad
U_{\sigma_n}(h_k,0)x_n\to\xi^k\in X$$
as $n\to\infty$ for every $k\in\N$, then we have the chain of equalities
$$
U_{\sigma}(h_k,0)\xi^{0}=\xi ^{k},\quad\forall k\in\N.
$$
\end{definition}

\begin{proposition}
\label{propano}
Let $U_{\sigma}(t,\tau)$ be asymptotically closed with respect to some sequence $h_k$
complying with Definition~\ref{DGM}, and let $T(h)$ be a continuous map\footnote{Since $\Sigma$ is compact,
we could equivalently ask $T(h)$ asymptotically closed with respect to $h_k$.} for all $h=h_k$.
Then the skew-product semigroup $S(h)$ is
also asymptotically closed with respect to $h_k$.
\end{proposition}

\begin{proof}
Assume that, for some
sequence $(x_n,\sigma_n)\in \boldsymbol{X}$, the convergence
$$
S(h_k)(x_{n},\sigma _{n})\rightarrow (\xi ^{k},\omega ^{k})\in \boldsymbol{X}
$$
holds for every $k\in\N$. By \eqref{semigeppo}, this translates into
$$
U_{\sigma_n}(h_k,0)x_n\to\xi^{k}\in X
\qquad\text{and}\qquad
T(h_k)\sigma_n\to\omega^{k}\in\Sigma.
$$
In particular,
$$\sigma_{n}\to\omega^{0},$$
and from the continuity of $T(h_k)$ we readily obtain
$$T(h_k)\omega^{0}=\omega^{k},$$
for every $k\in\N$.
Besides, appealing to the asymptotic closedness of
$U_{\sigma }(t,\tau )$, we also deduce the chain of equalities
$$
U_{\omega^{0}}(h_k,0)\xi^{0}=\xi^{k}.
$$
Hence, using \eqref{semigeppo} the other way around, we conclude that
$$
S(h_k)(\xi^{0},\omega^{0})=(\xi^{k},\omega^{k}).
$$
This proves the asymptotic closedness of $S(h)$.
\end{proof}

We are now ready to state the following
generalized version of Theorem~\ref{THMgazzola}.

\begin{theorem}
\label{theo3}
Let the family $U_{\sigma}(t,\tau )$ be uniformly asymptotically compact (or uniformly
totally dissipative if $X$ is complete). If $U_{\sigma}(t,\tau )$ is
asymptotically closed with respect to some sequence $h_k$ and $T(h_k)$ is continuous,
then
$$A_\Sigma=\bigcup_{\sigma\in\Sigma}{\mathcal K}_\sigma(0).$$
\end{theorem}

\begin{proof}
Indeed, we learn from Proposition~\ref{propano} that the skew-product
semigroup $S(h)$ on $\boldsymbol{X}$
is asymptotically closed with respect to $h_k$, hence Theorem~\ref{butano}
guarantees the full invariance of its
global attractor $\boldsymbol{A}$. At this point, the argument is the same as in the proof
of Theorem~\ref{THMgazzola}.
\end{proof}

\section{Differential Equations with Translation Compact Symbols}

\noindent
We finally apply the results to the study of a particular (although quite general)
class of nonautonomous
differential problems.
More precisely, we focus on
a single process $U_g(t,\tau)$ generated by a
nonautonomous differential equation on a Banach space $X$ of the form
\begin{equation}
\label{EUNO}
\frac{\dd}{\dd t}u(t)={\mathcal A}(u(t))+g(t),
\end{equation}
where ${\mathcal A}(\cdot)$ is a densely defined operator on $X$, and
$g$ (the symbol) is a function defined on $\R$ with values in some other normed space.
The problem is supposed to be well posed for every initial data $u_0\in X$ taken at
any initial time $\tau\in\R$.
\smallskip

We assume that
$g$ is {\it translation compact} as an element of a given metric space ${\mathfrak L}$.
By definition, this means
that the set of translates
$${\mathcal T}(g)=\big\{g(\cdot+h):\,h\in \R\big\}$$
is precompact in ${\mathfrak L}$.
The closure of ${\mathcal T}(g)$ in the space ${\mathfrak L}$
is called the {\it hull} of  $g$, and is denoted by ${\mathcal H}(g)$.

\begin{example}
Given a domain $\Omega\subset\R^N$, we consider the space
$${\mathfrak L}=L^2_{\rm loc}(\R;L^2(\Omega)).$$
Here, $f$ belongs to ${\mathcal H}(g)$
if there exists a
sequence $h_n\in\R$ such that
$$\lim_{n\to\infty}\,\int_a^b\|g(t+h_n)-f(t)\|^2_{L^2(\Omega)}\,\dd t=0,\quad\forall a>b.
$$
Several translation compactness criteria can
be found in \cite{CV}, also for different choices of the space ${\mathfrak L}$, such as
$L^p_{\rm loc}(\R;L^q(\Omega))$ or ${\mathcal C}_{\rm b}(\R,L^q(\Omega))$.
\end{example}

Then, for every $h\in\R$, we define the translation operator acting
on a vector-valued function $f$ on $\R$ as
$$[T(h)f](t)=f(h+t).$$
It is clear that the family $\{T(h)\}_{h\in\R}$ satisfies the axioms of
a {\it group} of operators on the compact space ${\mathcal H}(g)$.
We also assume that $T(h)$ is strongly continuous, i.e.\footnote{For
most general concrete spaces ${\mathfrak L}$, like those mentioned in the example,
the strong continuity of $T(h)$ is straightforward.}
$$T(h)\in{\mathcal C}({\mathcal H}(g),{\mathcal H}(g)),\quad\forall h\in\R.$$
In which case, it is apparent that
$$T(h){\mathcal H}(g)={\mathcal H}(g),\quad\forall h\in\R.$$
Along with the process $U_g(t,\tau)$ generated by \eqref{EUNO}, we also consider the family of processes
$$\{U_f(t,\tau)\}_{f\in{\mathcal H}(g)},$$
generated by the family of equations
\begin{equation}
\label{EDUE}
\frac{\dd}{\dd t}u(t)={\mathcal A}(u(t))+f(t),\quad f\in{\mathcal H}(g).
\end{equation}
Again, for any choice of the symbol $f\in{\mathcal H}(g)$,
the problem is supposed to be well posed for every initial data $u_0\in X$ taken at
any initial time $\tau\in\R$.
We note that the translation property \eqref{transl}, namely,
\begin{equation}
\label{ETRE}
U_{f}(h+t,h+\tau)=U_{T(h)f}(t,\tau),\quad\forall f\in{\mathcal H}(g),
\end{equation}
actually holds for every $h\in\R$.

\begin{remark}
Such a property reflects the obvious fact that
shifting the time in the initial data is the
same as shifting the time in the symbol.
\end{remark}

Hence, Theorem \ref{theo3} tailored for this particular framework reads as follows.

\begin{theorem}
\label{theo4}
Let the family $U_f(t,\tau)$ generated by \eqref{EDUE} be uniformly
totally dissipative. If it is also
asymptotically closed,
then
$$A_{{\mathcal H}(g)}=\bigcup_{f\in {\mathcal H}(g)}\big\{u(0):\,u(s)\text{ is a {\cbt} for }U_f(t,\tau)\big\}.$$
\end{theorem}

In fact, requiring a further continuity assumption, we can also provide a description of
the global attractor $A_{\{g\}}$ of the single process $U_g(t,\tau)$ generated by \eqref{EUNO}.

\begin{theorem}
\label{theo5}
Let the hypotheses of Theorem \ref{theo4} hold.
If in addition the map
$$f\mapsto U_f(t,\tau)u_0:{\mathcal H}(g)\to X$$
is continuous for every fixed  $t\geq \tau$ and $u_0\in X$, then
we have the equality
$$A_{\{g\}}=A_{{\mathcal H}(g)}.$$
\end{theorem}

\begin{proof}
The existence of $A_{{\mathcal H}(g)}$ implies that $A_{{\mathcal T}(g)}$ and $A_{\{g\}}$ exist too,
and
$$A_{\{g\}}\subset A_{{\mathcal T}(g)}\subset A_{{\mathcal H}(g)}.$$
In light of the additional continuity, we can apply Corollary~\ref{corry} to get
$A_{{\mathcal T}(g)}=A_{{\mathcal H}(g)}$. So, we are left to prove the equality $A_{\{g\}}=A_{{\mathcal H}(g)}$.
Indeed, for an arbitrary bounded set $B\subset X$, we infer from \eqref{ETRE} that
$$
\d_X(U_{T(h)g}(t,\tau )B,A_{\{g\}})=\d_X(U_{g}(h+t,h+\tau )B,A_{\{g\}}).
$$
This tells that the compact set $A_{\{g\}}$, in principle only contained in $A_{{\mathcal T}(g)}$,
is actually uniformly attracting for the family $\{U_f(t,\tau)\}_{f\in{\mathcal T}(g)}$,
hence coincides with its uniform global attractor $A_{{\mathcal T}(g)}$.
\end{proof}

\begin{remark}
An interesting open question is whether or not Theorem~\ref{theo5} remains valid {\it without}
the continuity hypothesis, lying only on the fact
that $f\mapsto U_f(t,\tau)u_0$ is a closed map.
\end{remark}

\section{A Concrete Application}

\noindent
Given a bounded domain $\Omega\subset\R^N$ ($N=1,2$) with smooth boundary $\partial\Omega$
(for $N=2$),
let $g$ be a translation compact
function in the space
$${\mathfrak L}=L^2_{\rm loc}(\R;L^2(\Omega)).$$
In this case, it is certainly true that
$$T(h)\in{\mathcal C}({\mathcal H}(g),{\mathcal H}(g)),\quad\forall h\in\R.$$
For any given initial time $\tau\in\R$, we consider the family of nonautonomous Cauchy problems
on the time-interval $[\tau,\infty)$ in  the unknown $u=u(t)$
depending on the external source $f\in{\mathcal H}(g)$
$$
\begin{cases}
u_{tt}+(1+u^2)u_t-\Delta u+u^3-u=f(t),\\
u_{|\partial\Omega}=0,\\
u(\tau)=u_0,\quad u_t(\tau)=v_0,
\end{cases}
$$
which can be viewed as a model of a vibrating string ($N=1$) or membrane ($N=2$) in a stratified
viscous medium.
Arguing as in \cite{GAT,PZ2D}, dealing with the same model for a time-independent $f$,
for every $\tau'>\tau$ and every
initial data $x=(u_0,v_0)$ in the weak energy space
$$X=H_0^1(\Omega)\times L^2(\Omega),
$$
there is a unique variational solution
$$u\in{\mathcal C}([\tau',\tau],H_0^1(\Omega))\cap {\mathcal C}^1([\tau',\tau],L^2(\Omega)).$$
Accordingly, the equation generates
a dynamical process
$$U_f(t,\tau):X\to X,$$
depending on the symbol $f\in{\mathcal H}(g)$. Repeating the proofs of \cite{GAT,PZ2D},
we can also find a compact uniformly (with respect to $f\in{\mathcal H}(g)$) attracting set.
Hence the family of processes is uniformly asymptotically compact, and
by Theorem~\ref{uac} we infer the existence
of the uniform global attractor $A_{{\mathcal H}(g)}$.
In order to understand the structure of the attractor, we shall distinguish two cases.

\smallskip
\noindent
$\bullet$ If $N=1$, repeating the proofs of \cite{GAT} one can show that the map
$$(x,f)\mapsto U_{f}(t,\tau)x$$
is continuous from $X\times {\mathcal H}(g)$ into $X$.
Thus both Theorem~\ref{theo4} and Theorem~\ref{theo5} apply, yielding
\begin{equation}
\label{paral}
A_{{\mathcal H}(g)}=\bigcup_{f\in {\mathcal H}(g)}
\big\{(u(0),u_t(0)):\,(u(s),u_t(s))\text{ is a {\cbt} for }U_f(t,\tau)\big\}
\end{equation}
along with the identity
$$A_{\{g\}}=A_{{\mathcal H}(g)}.$$

\smallskip
\noindent
$\bullet$ If $N=2$, the process is not strongly continuous.
Nonetheless, introducing the weaker space
$$W=L^2(\Omega)\times H^{-1}(\Omega),$$
one can prove the following
continuous dependence result, analogous to Proposition 2.5 of \cite{PZ2D}.

\begin{proposition}
\label{propclosed}
For every $t\geq\tau$, every $f_1,f_2\in {\mathcal H}(g)$ and every $R\geq 0$, we
have the estimate
$$
\|U_{f_1}(t,\tau)x_1-U_{f_2}(t,\tau)x_2\|_{W}
\leq C\e^{C(t-\tau)}\big[\|x_1-x_2\|_X+\|f_1-f_2\|_{L^2(t,\tau;L^2(\Omega))}\big],
$$
for some
$C=C(R)\geq 0$ and all initial data $x_1,x_2\in X$ of norm not exceeding $R$.
\end{proposition}

In other words, for every fixed $t\geq\tau$, we have the weaker continuity
$$(x,f)\mapsto U_{f}(t,\tau)x\in{\mathcal C}(X\times {\mathcal H}(g),W).$$
This is enough to infer that the map
$$(x,f)\mapsto U_{f}(h,0)x:X\times {\mathcal H}(g)\to X$$
is closed for every $h\geq 0$.
We conclude from Theorem~\ref{theo4} that $A_{{\mathcal H}(g)}$ fulfills the
same characterization \eqref{paral} of the case $N=1$.

\begin{remark}
If the function $g$ is periodic, i.e.\
$$g(\cdot+p)= g(\cdot)\quad\text{for some}\,\,
p>0,$$
then we have the trivial equality
$${\mathcal H}(g)={\mathcal T}(g)=\big\{g(\cdot+h):\,0\leq h<p\big\},$$
providing at once the identity $A_{\{g\}}=A_{{\mathcal H}(g)}$.
Moreover, it is known that the uniform global attractor
of a periodic process coincides with the nonuniform (with respect to the initial time $\tau\in\R$)
one. More details can be found in \cite{CV3,CV4}.
\end{remark}



\end{document}